\documentclass[12pt,reqno]{amsart}

\usepackage{color}
\usepackage{amsmath, amsthm, amssymb}
\usepackage{amsfonts}
\usepackage{enumitem}
\usepackage[ansinew]{inputenc}
\usepackage[dvips]{epsfig}
\usepackage{graphicx}
\usepackage[english]{babel}
\usepackage{hyperref}
\usepackage{bbm}
\theoremstyle{plain}
\newtheorem{thm}{Theorem}[section]
\newtheorem{cor}[thm]{Corollary}

\newtheorem{prop}[thm]{Proposition}

\theoremstyle{definition}

\theoremstyle{remark}
\newtheorem{rem}[thm]{Remark}

\numberwithin{equation}{section}

\newcommand{\de}{\partial}

\newcommand{\R}{\mathbb{R}}

\newcommand{\N}{\mathbb{N}}

\newcommand{\eps}{\varepsilon}

\newcommand{\average}{{\mathchoice {\kern1ex\vcenter{\hrule height.4pt
width 6pt depth0pt} \kern-9.7pt} {\kern1ex\vcenter{\hrule
height.4pt width 4.3pt depth0pt} \kern-7pt} {} {} }}

\def\R{\mathbb{R}}

\begin{document}

\title{On the obstacle problem for the 1D wave equation}

\author{Xavier Fern\'andez-Real}
\address{ETH Z\"urich, Department of Mathematics, R\"amistrasse 101, 8092 Z\"urich, Switzerland}
\email{xavierfe@math.ethz.ch}

\author{Alessio Figalli}
\address{ETH Z\"urich, Department of Mathematics, R\"amistrasse 101, 8092 Z\"urich, Switzerland}
\email{alessio.figalli@math.ethz.ch}



\maketitle

\begin{center}
{\it A Sandro Salsa per il suo $50_{14}$ compleanno, con amicizia ed ammirazione}
\end{center}

\begin{abstract}
Our goal is to review the known theory on  the one-dimensional obstacle problem for the wave equation, and to discuss some extensions. We introduce the setting established by Schatzman within which existence and uniqueness of solutions can be proved, and we prove that (in some suitable systems of coordinates) the Lipschitz norm is preserved after collision. As a consequence, we deduce that solutions to the obstacle problem (both simple and double) for the wave equation have bounded Lipschitz norm at all times. Finally, we discuss the validity of an explicit formula for the solution that was found by Bamberger and Schatzman.
\end{abstract}

\vspace{5mm}

\section{Introduction}

\subsection{The obstacle problem} Consider an infinite vibrating string represented by its transversal displacement, $u(x, t)\in \R$, with $x \in \R$ and $t\in [0, \infty)$, with initial conditions given by 
\[
\left\{
\begin{array}{rcll}
u(x, 0) & = &u_0(x)& \textrm{ for } x\in \R\\
u_t(x, 0) & = &u_1(x)& \textrm{ for } x\in \R.
\end{array}
\right.
\]
Suppose that the string is vibrating freely, but it is restricted to remain above a certain given obstacle, which we denote $\varphi = \varphi(x)$ (in particular, we assume $u_0\ge \varphi$). Thus, the vibrating string $u$ fulfills the homogeneous wave equation whenever $u >\varphi$:
\begin{equation}
\label{eq.we}
\square u := u_{tt}-u_{xx}  = 0\quad\textrm{in}\quad \{u >\varphi\}. 
\end{equation}
In order to get a closed system to describe this phenomenon, one also needs to provide information regarding the interaction between the string and the obstacle. As we will explain, a natural condition is to assume that the string \emph{bounces} elastically at the point of contact, in the sense that the sign of the velocity is instantly flipped. That is, if $(x_\circ, t_\circ)$ is a contact point (i.e., $u(x_\circ, t_\circ)  = \varphi(x_\circ)$), then 
\begin{equation}
\label{eq.refl}
u_t(x_\circ, t_\circ^+) = - u_t(x_\circ, t_\circ^-),
\end{equation}
where $t_\circ^\pm$ denotes taking limits $t_\circ\pm \eps$ as $\eps \downarrow 0$. 

Let us assume that the obstacle is given by a wall, so that we can take $\varphi \equiv 0$. This problem was first studied by Amerio and Prouse in \cite{AP75} in the finite string case (with fixed end-points), constructing  a solution ``by hand'' by following the characteristic curves and extending the initial condition through the lines of influence.  This proved existence and uniqueness in a ``non-standard'' class of solutions, by means of very intuitive methods.

A similar approach was used by 	Citrini in \cite{Cit75} to study properties of solutions (in particular, the number of times the obstacle is hit) assuming that collisions can lose energy and be either inelastic, partially elastic, or completely elastic, by replacing equation \eqref{eq.refl} with  $u_t(x_\circ, t_\circ^+) = - hu_t(x_\circ, t_\circ^-)$, where $h\in [0, 1]$ denotes the loss of energy in each collision. In this case, its clear that if $h < 1 $ then the local kinetic energy, namely $u_t^2$, is no longer preserved. 

In this work we focus on the approach introduced by Schatzman in \cite{Sch80}, where existence and uniqueness of solutions was proved for a more natural class of solutions (with initial conditions $u_0\in W^{1,2}$ and $u_1\in L^2$). To do so, instead of proceeding with a variational proof, Schatzman  explicitly expresses the solution in terms of the free wave with the same initial data: she adds to the free wave an appropriate measure convoluted with the fundamental solution of the wave operator, in order to ensure that the specular reflection holds. We show that, at least in the ``right'' system of coordinates, solutions built in this way preserve the Lipschitz constant in space-time (see Corollary~\ref{cor.Lip}). This immediately yields that, even if one considers a second obstacle acting from above (say, $\bar\varphi \equiv 1$, so that $\varphi\le u \le\bar \varphi$ for all times), a suitable Lipschitz-type norm of the solution is constant in time. In particular, solutions remain uniformly Lipschitz independently of the number of collisions (see Proposition~\ref{prop.dlip}). 

It is important to notice that, in \cite{BS83}, Bamberger and Schatzman studied a penalized problem and proved the convergence of solutions to the solution of the obstacle problem for the wave equation, with general obstacles. In that paper they also gave a simple explicit formula for the solution to the obstacle problem when the obstacle is zero, but unfortunately their formula is not correct, as we shall discuss in Section \ref{sect:counterex}.

\subsection{Other problems with constrains} The works mentioned so far cover most of the literature regarding the obstacle problem in the context of the wave equation, which is mostly restricted to the one-dimensional case where characteristic equations can be extensively used. The lack of results in higher dimensions or more general obstacles could be associated to the need of a more precise model, see Subsection~\ref{ss.gen} below. We hope that this paper will be of stimulus for investigating these more general problems.

There are also other problems with constrains within the context of hyperbolic equations that seem to exhibit cleaner behaviors. In particular, the \emph{thin} obstacle problem for the wave equation is a simple approximation to the general dynamical Signorini problem. In this case, one looks for solutions to the wave equation with a unilateral constrain posed on a lower-dimensional manifold. The problem was original studied by Amerio in \cite{Ame76} and Citrini in \cite{Cit77}, and later by Schatzman in \cite{Sch80b}, where existence and uniqueness is proved: contrary to the obstacle problem presented above, the conservation of energy is a direct consequence of the equations of motion and does not need to be imposed to have a well-posed problem. Later, Kim studied the problem in a variational way in \cite{Kim89}, and more recently, even in a non-deterministic approach in \cite{Kim10}. 

\subsection{On the obstacle and the model}
\label{ss.gen}
As mentioned above, we restrict our attention to constant obstacles ($\varphi \equiv 0$). As a direct consequence, our result directly applies if one considers linear obstacles $\varphi = ax + b$, since $\square \varphi \equiv 0$ and the specular reflection of the wave when hitting obstacle (in the vertical direction) is still preserved. Alternatively, the function $u -\varphi$ still presents conservation of energy (see \eqref{eq.consene} below).

Notice, however, that time-dependent obstacles might present added difficulties. For instance, even the simple case $\varphi = t$ does not directly follow from our analysis, since specular reflection is not preserved  if one replaces $u$ by $u-t$.

In \cite{Sch80} Schatzman studies the case of general convex obstacles ($\varphi''\ge 0$). This condition is necessary to use the techniques presented there: when the obstacle is not convex, one could have infinitely many collisions accumulating in space-time to a single point. The convexity of the obstacle ensures that collisions occur only once at each point in the infinite string case (namely $x 
\in \R$), and that they do not accumulate in time in the finite string case (namely $x \in I$ where $I$ is a bounded interval, and $u$ is fixed on $\partial I$). In fact, in the infinite string case, one usually expects solutions to diverge to infinity as time goes by (see Remark \ref{rem.inf}). 

It is currently unclear how the reflection condition \eqref{eq.refl} should be modified in the general obstacle case. Indeed, one would expect reflections to occur perpendicularly to the obstacle, rather than vertically. Vertical reflections come from a small oscillation assumption, which is also the same assumption used to derive the wave equation as a model for a vibrating string (see for instance \cite[Chapter 1.4.3]{PR05}). In this sense it is not reasonable to assume non-flat obstacles, and one may wonder whether one can prove existence/uniqueness results assuming {small} initial data and an obstacle with {small} oscillations. 

Alternatively, if one wants to impose reflections perpendicular to the obstacle (thus hoping to avoid accumulation of collisions), one would need to consider a rotation invariant equation (in the graph space $(x, u(x, t)) \in \R^2$) instead of the wave equation (compare for instance \cite[Equation (1.28) vs (1.29)]{PR05}).
Investigating these modeling questions is a very interesting problem, and we hope that this paper will be a starting point to motivate this beautiful line of research.

Finally, we should mention that in this paper we consider the model of an infinite string. Nonetheless, the results presented here can be easily extended to the case of a finite string with fixed end-points, thanks  to the locality of our methods. 
\\[0.5cm]
\noindent
{\bf Acknowledgement:} This work has received funding from the European Research Council (ERC) under the Grant Agreement No 721675.

\section{Schatzman's existence and uniqueness}

We consider the zero obstacle case $\varphi \equiv 0$. Notice that, in the sense of distributions, the support of $\square u$ is contained in $\{u > 0\}$, and $\square u \ge 0$. On the other hand, $u \ge 0$ by assumption. Our problem can then be written as 
\begin{equation}
\label{eq.pb}
\left\{
\begin{array}{rcll}
\min\{\square u, u\} & = & 0& \textrm{ in } \R\times [0, \infty)\\
u(\cdot, 0) & = &u_0& \textrm{ in } \R\\
u_t(\cdot, 0) & = &u_1& \textrm{ in } \R.
\end{array}
\right.
\end{equation}
\begin{rem}
Notice that, formally, the formulation above is analogous to the formulation of the parabolic (or elliptic) obstacle problem, which can written as 
\[
\min\{Lu, u-\varphi\} = 0
\]
for the corresponding operator (say, $L = \de_t - \Delta$ or $L = -\Delta$). In the current situation, however, an extra condition will need to be imposed.
\end{rem}
We consider initial data such that $u_0\in W^{1, 2}_{\rm loc}(\R)$ (in particular, it is continuous) and $u_1\in L^2_{\rm loc}(\R)$, and we are interested in the existence and uniqueness of solutions in the  natural class 
\[
u\in L^\infty_{{\rm loc}, t} ((0, \infty); W^{1, 2}_{{\rm loc}, x} (\R))\cap W^{1,\infty}_{{\rm loc}, t} ((0, \infty); L^{ 2}_{{\rm loc}, x} (\R)).
\]
That is, for any compact $K\subset \subset \R$, and any $T > 0$,
\[
\int_K\left\{\left|u(x, t)\right|^2+\left|u_x(x, t)\right|^2+\left|u_t(x, t)\right|^2\right\}\, dx \le C(T, K)<\infty\quad\textrm{for a.e.  }t\in (0, T),
\]
for some constant $C(T, K)$ independent of $t$. 

As mentioned before, we need to provide information regarding the type of reflection we are expecting. That is, \eqref{eq.pb} is not enough to ensure a unique solution to our problem. In this case, the notion introduced by Schatzman  imposes a local energy conservation (corresponding to an elastic collision) in the form 
\begin{equation}
\label{eq.consene}
{\rm div}_{x, t} (-2u_x u_t, u_x^2 + u_t^2) = \frac{d}{dx}\left(-2u_x u_t\right) + \frac{d}{dt} \left(u_x^2 + u_t^2\right) = 0\quad\textrm{in}\quad \R\times (0, \infty),
\end{equation}
and needs to be understood in the sense of distributions. A posteriori, this notion implies that solutions to our problem are elastically reflected, as in \eqref{eq.refl}, which is well-defined almost everywhere. Thus, the equations describing our problem are \eqref{eq.pb}-\eqref{eq.consene}.

The main theorem in \cite{Sch80} is then the following:
\begin{thm}[{\cite[Theorem IV.1]{Sch80}}]
\label{thm.sch}
Let $u_0\in W^{1, 2}_{\rm loc}(\R)$ and $u_1\in L^2_{\rm loc}(\R)$. Assume that $u_0 \ge 0$, and that $u_1 \ge 0$ a.e. in $\{u_0 = 0\}$. Then, there exists a unique solution $u\in L^\infty_{{\rm loc}, t} ((0, \infty); W^{1, 2}_{{\rm loc}, x} (\R))\cap W^{1,\infty}_{{\rm loc}, t} ((0, \infty); L^{ 2}_{{\rm loc}, x} (\R))$ to
\begin{equation}
\label{eq.pb2_s}
\left\{
\begin{array}{rcll}
\min\{\square u, u\} & = & 0& \textrm{ in } \R\times [0, \infty)\\
u(\cdot, 0) & = &u_0& \textrm{ in } \R\\
u_t(\cdot, 0) & = &u_1& \textrm{ a.e. in } \R.
\end{array}
\right.
\end{equation}
such that \eqref{eq.consene} holds in the sense of distributions. 
\end{thm}

\subsection{Construction of the solution}

In order to prove the previous result, Schatzman builds an explicit solution in terms of the free wave equation with the same initial data. Let us denote $w$ the solution to
\begin{equation}
\label{eq.fw}
\left\{
\begin{array}{rcll}
\square w & = & 0& \textrm{ in } \R\times [0, \infty)\\
w(\cdot, 0) & = &u_0& \textrm{ in } \R\\
w_t(\cdot, 0) & = &u_1& \textrm{ a.e. in } \R.
\end{array}
\right.
\end{equation}
If $\mathcal{E}$ denotes the fundamental solution to the one-dimensional wave equation,  namely
\begin{equation}
\label{eq:E}
\mathcal{E}(x, t) = \textstyle{\frac12} \mathbbm{1}_{\{t\ge |x|\}},
\end{equation}
then, by d'Alembert's formula, $w$ can written as 
$$
w = \de_t(\mathcal{E}\ast_x u_0) + \mathcal{E}\ast_x u_1,
$$
or more explicitly
$$
w(t,x)=\frac{u_0(x-t)+u_0(x+t)}{2}+\frac12\int_{x-t}^{x+t}u_1(s)\,ds\qquad \text{for all }(x,t)\in \R\times [0,\infty).
$$
(see for instance \cite[Chapter 4]{PR05}).

Let us also denote by $T_{x, t}^-$ the cone of dependence of the point $(x, t)$, namely
\[
T_{x, t}^- := \{(x', t') \in \R\times[0, \infty) : |x-x'|\le t-t'\},
\]
and by $T_{x, t}^+$ the cone of influence of the point $(x, t)$, that is
\[
T_{x, t}^+ := \{(x', t') \in \R\times[0, \infty) : |x-x'|\le t'-t\}.
\]
One can note that the solution $u$ coincides with $w$ outside the domain of influence of the set of points where $w < 0$. More precisely, if we denote 
\[
E := \overline{\{(x, t) \in \R\times[0, \infty) : w(x, t) < 0 \}}\qquad\textrm{and}\qquad I = \bigcup_{(x, t)\in E} T_{x, t}^+,
\]
then $u \equiv w $ in $(\R\times[0, \infty))\setminus I$ (since inside $(\R\times[0, \infty))\setminus I$ the solution $u$ has not touched the obstacle yet, and so it behaved as a free wave).
Thanks to this remark, one is left with building the solution $u$ inside the domain $I$.  (See Figure~\ref{fig.3} for a representation of such regions.)

\begin{figure}
\centering
\includegraphics{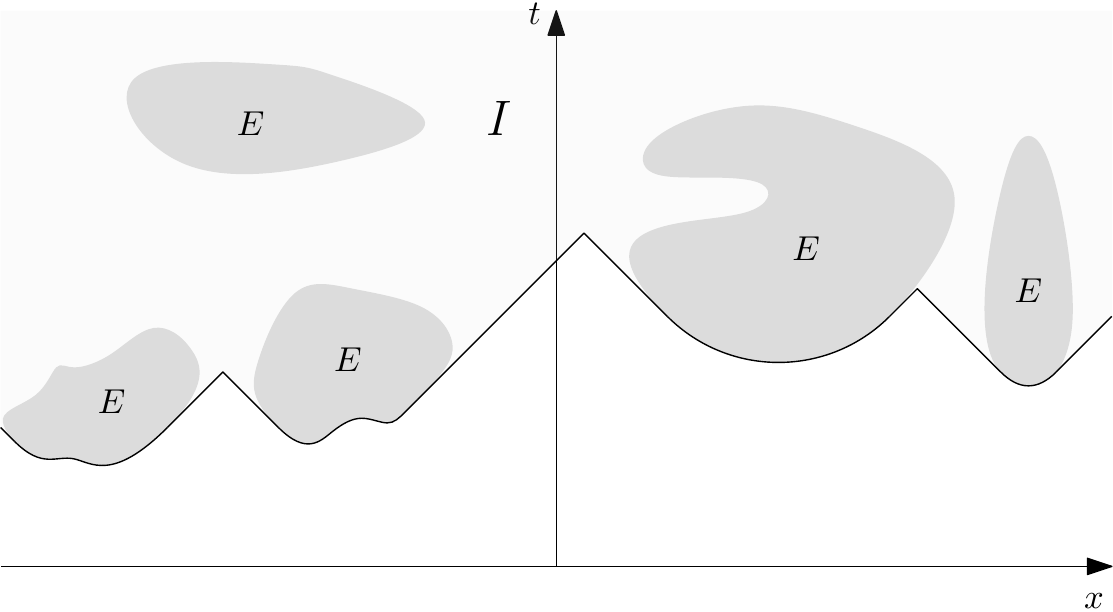}
\caption{Representation of the regions $E$ where $w < 0$ and its domain of influence, $I$. }
\label{fig.3}
\end{figure}

To do that, one first notice that $I$ coincides with the epigraph of a Lipschitz function $\tau:\R\to [0, \infty)$ of Lipschitz constant 1,  so that 
\[
I = \{(x, t) \in \R\times[0, \infty) : t \ge \tau(x)\}\quad\textrm{and}\quad |\tau'|\le 1\textrm{  in  } \R
\] 
(see \cite[Proposition II.3]{	Sch80}). Moreover, the \emph{active} contact points (that is, those points where the solution is not just grazing the obstacle) are determined by the graph of $\tau$ whenever $|\tau'|< 1$ (thus, $|\tau'| < 1$ implies $w(x, \tau(x)) = 0$). 

It is proven in \cite{Sch80} that the solution to \eqref{eq.pb2_s}-\eqref{eq.consene} is given by 
\begin{equation}
\label{eq.sol}
u = w + \mathcal{E}\ast \mu(w),
\end{equation}
where $\mu(w)$ is the measure  defined by the formula
\[
\langle \mu(w), \psi\rangle = -2\int_{\{x : \tau(x) > 0\}} (1-\tau'(x)^2) \,w_t(x, \tau(x) ) \,\psi(x, \tau(x))\, dx\qquad \forall\,\psi \in C_c(\R\times [0,+\infty)),
\]
and $\mathcal{E}$ is as in \eqref{eq:E}.

Then, Theorem~\ref{thm.sch} can be proved using the representation \eqref{eq.sol} by checking that such solution fulfils all the hypotheses.

In particular, it is observed that in the infinite string case, the obstacle is touched at most once at every point $x\in \R$ (that is, for any $x \in \R$, there is at most one time $t \in [0,\infty)$ such that $u(x, t) = 0$ and $u_t(x, t) < 0$). 

\subsection{Formal derivation of \eqref{eq.sol}} Let us formally show that the formula \eqref{eq.sol} solves the obstacle problem for the wave equation, in the sense \eqref{eq.we}-\eqref{eq.refl}, for $\varphi \equiv 0$ (for the actual proof, we refer the reader to \cite[Theorem IV.2]{Sch80}). 

Let $v(x, t) = \left[\mathcal{E}*\mu(w)\right] (x, t)$, so that 
\[
u = w + v.
\]
Since $\mathcal{E}$ is the fundamental solution to the one-dimensional wave equation, $u$ solves the wave equation outside
$$
{\rm supp}(\mu(w)) = \{(x, \tau(x)) : |\tau'(x)| < 1\}.
$$ Also, one can notice that $v \equiv 0$ in $\{t < \tau(x)\}$, and ${\rm supp}(\mu(w))\subset \{u \equiv 0\}$. 

We now (formally) show that the measure $\mu(w)$ ensures that the velocity of $u$ is instantly flipped when it hits the obstacle. Since $v \equiv 0$ in $\{t < \tau(x)\}$, it is enough to show that if $(x_\circ, t_\circ) = (x_\circ, \tau(x_\circ)) \in \{w \equiv 0\}$, then
\[
\lim_{t \downarrow t_\circ} v_t (x_\circ, t) = -2 w_t(x_\circ, t_\circ). 
\]
We want to compute 
\[
\lim_{t \downarrow t_\circ} \lim_{dt \downarrow 0} \frac{v (x_\circ, t+dt) - v(x_\circ, t)}{dt}.
\]
From the definition of $v$ and $\mu(w)$, we have that 
\[
v (x_\circ, t+dt) - v(x_\circ, t) = -\int_{A_{x_\circ, t, dt}} (1-\tau'(x)^2) w_t(x, \tau(x))\, dx,
\]
where 
\[
A_{x_\circ, t, dt} := \{x : t-\tau(x) < |x - x_\circ| < t+dt-\tau(x)\}.
\]
If we assume that $w\in C^1$, then when $t\downarrow t_\circ$ and $dt\downarrow 0$ we have that
\begin{equation}
\label{eq.vdiff}
v (x_\circ, t+dt) - v(x_\circ, t) \approx - |A_{x_\circ, t, dt}| (1-\tau'(x_\circ)^2)w_t(x_\circ, t_\circ). 
\end{equation}
On the other hand, let us denote 
\[
A_{x_\circ, t, dt} = L_A\cup R_A,
\]
where 
\[
L_A:= \{x\in A_{x_\circ, t, dt} : x  < x_\circ\}\quad \textrm{and}\quad R_A:= \{x\in A_{x_\circ, t, dt} : x  > x_\circ\}
\]
(see Figure~\ref{fig.A}).
\begin{figure}
\centering
\includegraphics{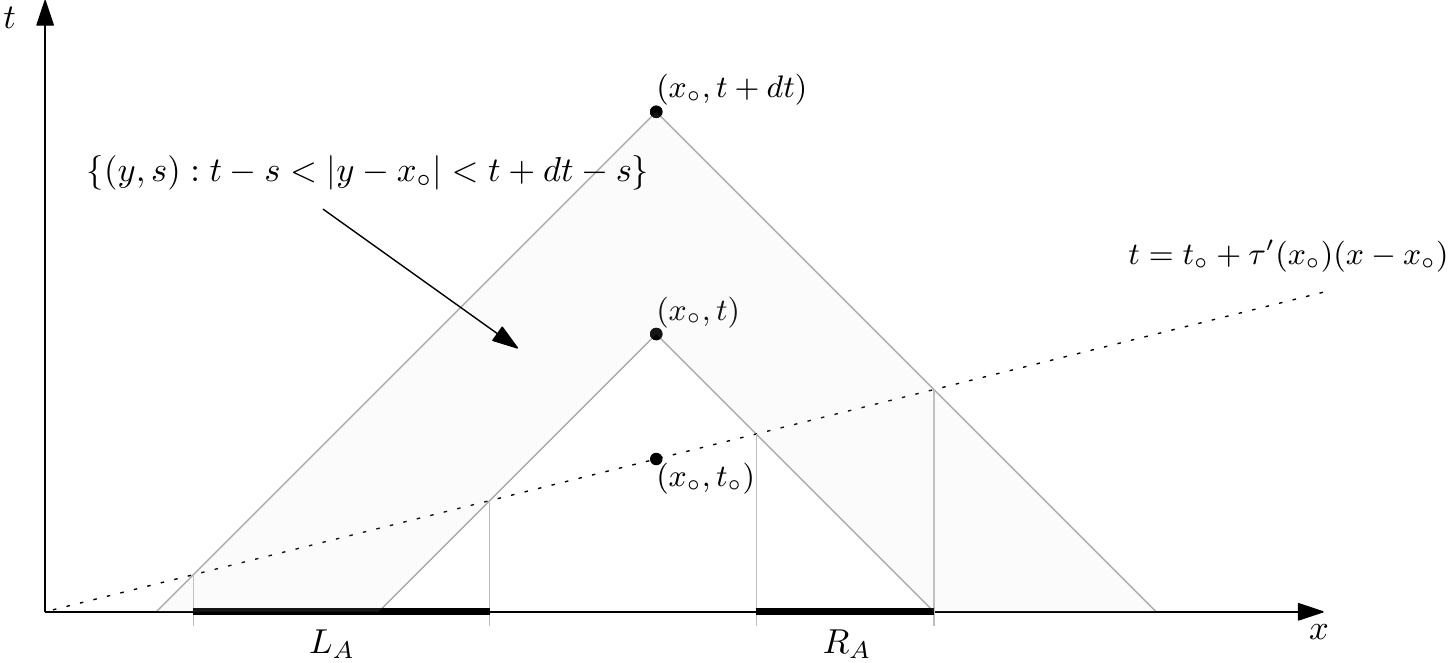}
\caption{Representation of the sets $L_A$ and $R_A$.}
\label{fig.A}
\end{figure}
Note that, as $t\downarrow t_\circ$ and $dt\downarrow 0$, $\tau'$ is essentially constant inside $A_{x_\circ,t,dt}$. Hence, a simple geometric argument yields that 
\[
|L_A| = \frac{dt}{1-\tau'(x_\circ)}+o(dt)\qquad\textrm{and}\qquad |R_A| = \frac{dt}{1+\tau'(x_\circ)}+o(dt).
\]
Thus, 
\[
|A_{x_\circ, t, dt}|  = |L_A|+|R_A| = 2\frac{dt}{1-\tau'(x_\circ)^2}+o(dt),
\]
and using \eqref{eq.vdiff} we reach 
\[
\lim_{t \downarrow t_\circ} v_t (x_\circ, t) = \lim_{t \downarrow t_\circ} \lim_{dt \downarrow 0} \frac{v (x_\circ, t+dt) - v(x_\circ, t)}{dt} = -2w_t(x_\circ, t_\circ)
\]
as desired.

This proves that the formula \eqref{eq.sol} guarantees that $u$ flips the velocity when hitting the obstacle. In order to show that \eqref{eq.sol} gives the solution to the obstacle problem for the wave equation, one still needs to show that $u\ge 0$ at all times. 
This is proved in \cite[Proof of Theorem IV.2]{Sch80}, using the characteristic variables (see the next section, for a definition of the characteristic variables). We refer the interested reader to the original proof.

\section{Conservation of the Lipschitz norm}
The goal of this section is to describe how the Lipschitz regularity of the solution is affected by the reflection.

Let $x\mapsto \sigma(x)$ be a non-negative Lipschitz function with $|\sigma'|\le 1$, and define 
\[
v = w + \mathcal{E}\ast \mu(w, \sigma),
\]
where $w$ solves \eqref{eq.fw}, and $\mu(w, \sigma)$ is given by
\[
\langle \mu(w,\sigma), \psi\rangle = -2\int_{\{x : \sigma(x) > 0\}} (1-\sigma'(x)^2)\, w_t(x, \sigma(x) )\, \psi(x, \sigma(x))\, dx \qquad \forall\,\psi \in C_c(\R\times [0,+\infty)).
\]
In particular, when $\sigma = \tau$ we recover the solution given by Schatzman (see \eqref{eq.sol}). 

Let us consider the characteristic variables\footnote{This change of variables is very standard in the theory of the 1D wave equation. The interested reader may consult \cite[Chapter 4]{PR05} for an overview of this theory.}
\[
\xi := \frac{x+t}{\sqrt{2}}\qquad\textrm{and}\qquad \eta := \frac{-x+t}{\sqrt{2}},
\]
and, as in \cite{Sch80}, we use the tilde notation to denote functions in the characteristic coordinates (e.g. $\tilde v(\xi, \eta) = v(x, t)$). We are interested in explicitly writing the derivatives $\tilde v_\xi$ and $\tilde v_\eta$ in the characteristic coordinates. 

Since $\sigma$ has Lipschitz constant 1, we can express it as a graph in the $(\xi, \eta)$-variables in two ways: either as $(\xi, Y(\xi))$ or as $(X(\eta), \eta)$.
In other words, 
\[
\eta   \in Y(\xi) ~\Leftrightarrow ~\frac{\xi+\eta}{\sqrt{2}} = \sigma\left(\frac{\xi-\eta}{\sqrt{2}}\right) ~\Leftrightarrow ~\xi\in X(\eta).
\]
Note that, at some points, the value of $Y$ or $X$ may not be uniquely determined, since these functions may have a vertical segment in their graphs. In this case we shall always refer to $Y(\xi)$ and $X(\eta)$ as the unique upper-semicontinuous representatives. In other words, geometrically, when one take a point $(\xi,\eta)$ and draws the two lines with slope $\pm 1$ to find $(X(\eta),\eta)$ and $(\xi,Y(\xi))$, we always choose $X(\eta)$ and $Y(\xi)$ to be the first point where these lines hit the graph of $\sigma$ (see  Figure~\ref{fig.1}).
It is important to notice that $X = Y^{-1}$.

\begin{figure}
\centering
\includegraphics{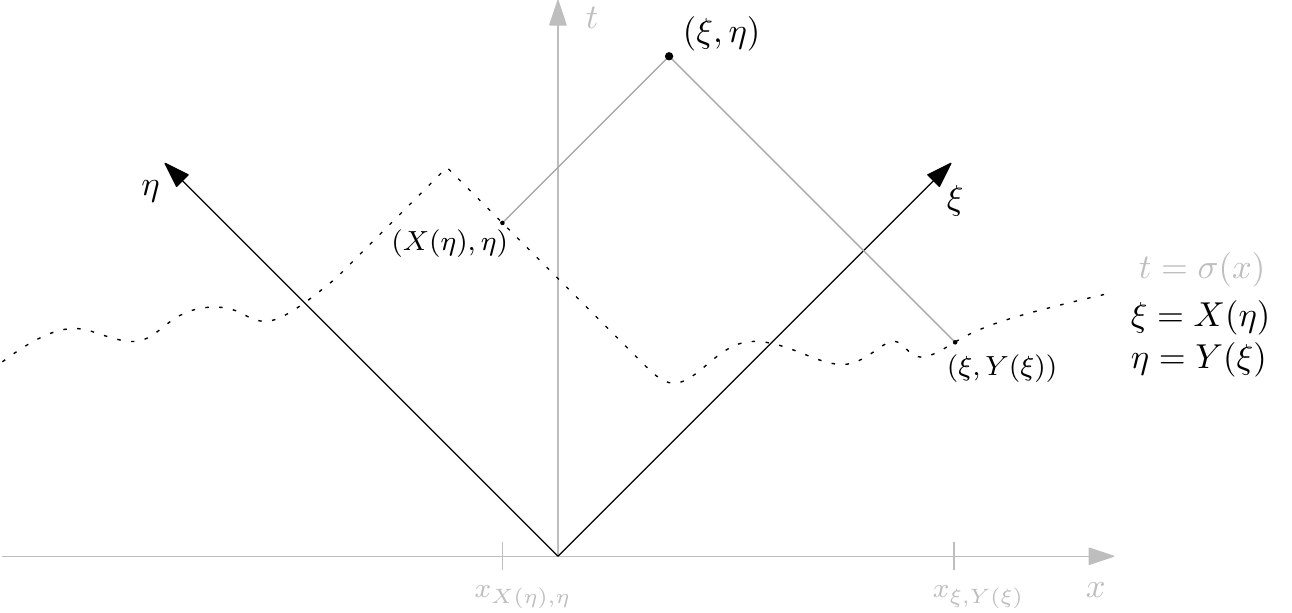}
\caption{Representation of $\sigma$ in characteristic coordinates. }
\label{fig.1}
\end{figure}


By standard transport along the characteristics one can easily show that
\[
\begin{array}{ll}
\tilde v_\xi(\xi, \eta)  = \tilde w_\xi(\xi, Y(\xi))&\quad\textrm{a.e. in}\quad \{\xi < X(\eta)\},\\
\tilde v_\eta(\xi, \eta)  = \tilde w_\eta(X(\eta), \eta)&\quad\textrm{a.e. in}\quad \{\eta  < Y(\xi)\},
\end{array}
\]
that is, derivatives in $\xi$ are transported along lines of the form $\{\xi \equiv {\rm constant}\}$ before the \emph{collision} occurs (i.e., in the region where $v$ is a free wave, before the convolution term appears), and analogously for derivatives in $\eta$.
(This is just a standard consequence of the explicit formula for solutions to the wave equations, that in the $(\xi,\eta)$-variables reads as $\tilde w_{\xi\eta}=0$.)

On the other hand, derivatives in the remaining region are computed in \cite[(III.37) and (III.38)]{Sch80} and are equal to 
\[
\begin{array}{ll}
\tilde v_\xi(\xi, \eta)  = \tilde w_\xi(\xi, Y(\xi))(1-g(\xi))-\tilde w_\eta(\xi, Y(\xi))g(\xi)&\quad\textrm{a.e. in}\quad \{\xi \ge X(\eta)\}\\
\tilde v_\eta(\xi, \eta)  = -\tilde w_\xi(X(\eta), \eta)h(\eta)+\tilde w_\eta(X(\eta), \eta)(1-h(\eta))&\quad\textrm{a.e. in}\quad \{\eta\ge Y(\xi)\},
\end{array}
\]
where 
\[
g(\xi) = \frac{-2Y'(\xi)}{1-Y'(\xi)}\mathbbm{1}_{\{\xi+Y(\xi)>0\}}\quad\textrm{and}\quad h(\eta) = \frac{-2X'(\eta)}{1-X'(\eta)}\mathbbm{1}_{\{X(\eta)+\eta>0\}}
\]
(notice that the indicator functions above represent the region where $\sigma > 0$).  Equivalently, if we denote 
\[
x_{\xi, \eta} := \frac{\xi-\eta}{\sqrt{2}},
\]
we can express the previous derivatives in terms of $\sigma$ instead of the functions $X$ and $Y$:
\[
\tilde v_\xi(\xi, \eta)  = \tilde w_\xi(\xi, Y(\xi))\sigma'(x_{\xi,Y(\xi)})-\tilde w_\eta(\xi, Y(\xi))(1-\sigma'(x_{\xi,Y(\xi)}))~\textrm{ a.e. in }\{\xi \ge X(\eta)\}
\]
and 
\[
\tilde v_\eta(\xi, \eta)  = -\tilde w_\xi(X(\eta), \eta)(1+\sigma'(x_{X(\eta),\eta}))-\tilde w_\eta(X(\eta), \eta)\sigma'(x_{X(\eta),\eta})\textrm{ a.e. in }\{\eta\ge Y(\xi)\}.
\]
We have ignored the region where $\sigma = 0$ in this case, because a posteriori we will use $\sigma = \tau$ and, thanks to the assumptions in Theorem~\ref{thm.sch}, $\tau (x)> 0$ for a.e. $x\in \R$.

\begin{prop}
\label{prop.Lip}
Under the same assumptions as in Theorem~\ref{thm.sch}, the solution of the obstacle problem for the wave equation $u$ fulfills that 
\[
|u_\xi|_\eta \equiv |u_\eta|_\xi \equiv 0
\]
a.e. in $\xi+\eta \ge 0$. In particular, 
\[
|u_\xi(x, t)| = |u_\xi(x+t, 0)| \qquad\textrm{and}\qquad |u_\eta(x, t)| =  |u_\eta(x-t, 0)| 
\]
a.e. in $t \ge 0$. 
\end{prop}
\begin{proof}
Thanks to the discussion above, by taking $\sigma = \tau$ we are considering the solution by Schatzman. Thus, we know that 
\[
\begin{array}{ll}
\tilde u_\xi(\xi, \eta)  = \tilde w_\xi(\xi, Y(\xi))&\textrm{a.e. in }\{\xi < X(\eta)\}\\
\tilde u_\xi(\xi, \eta)  = \tilde w_\xi(\xi, Y(\xi))\tau'(x_{\xi,Y(\xi)})-\tilde w_\eta(\xi, Y(\xi))(1-\tau'(x_{\xi,Y(\xi)}))&\textrm{a.e. in }\{\xi \ge X(\eta)\},
\end{array}
\]
and
\[
\begin{array}{l}
\tilde u_\eta(\xi, \eta)  = \tilde w_\eta(X(\eta), \eta)\hspace{7.25cm}\textrm{ a.e. in } \{\eta  < Y(\xi)\}~\\
\tilde u_\eta(\xi, \eta)  = -\tilde w_\xi(X(\eta), \eta)(1+\tau'(x_{X(\eta),\eta}))-\tilde w_\eta(X(\eta), \eta)\tau'(x_{X(\eta),\eta})\\
\hfill\textrm{ a.e. in }\{\eta\ge Y(\xi)\}.
\end{array}
\]
We now notice that,  whenever $|\tau'(x)| < 1$ (that is, at collision points), $\tau'(x)$ can be expressed in terms of $w_x(x, \tau(x))$ and $w_t(x, \tau(x))$ as 
\[
\tau'(x) = -\frac{w_x(x,\tau(x))}{w_t(x,\tau(x))}\qquad\textrm{if}\quad|\tau'(x)|< 1.
\]
Using that $w_x  = \frac{1}{\sqrt{2}}\left(w_\xi-w_\eta\right) $ and $w_t  = \frac{1}{\sqrt{2}}\left(w_\xi+w_\eta\right) $ we obtain that
\[
\tau'(x_{\xi,\eta}) = -\frac{\tilde w_\xi(\xi,\eta)- \tilde w_\eta(\xi,\eta)}{\tilde w_\xi(\xi,\eta)+ \tilde w_\eta(\xi,\eta)}\qquad\textrm{if}\quad\frac{\xi+\eta}{\sqrt{2}} = \tau\left(\frac{\xi-\eta}{\sqrt{2}}\right)\quad\textrm{and}\quad|\tau'(x_{\xi,\eta})|< 1.
\]

Thus,
\[
\begin{array}{ll}
\tilde u_\xi(\xi, \eta)  = \tilde w_\xi(\xi, Y(\xi))&\quad\textrm{ a.e. in }\{\xi < X(\eta)\}\cup\{\xi \ge X(\eta),\tau'(x_{\xi,Y(\xi)}) = 1\}\\
\tilde u_\xi(\xi, \eta)  = -\tilde w_\xi(\xi, Y(\xi))&\quad\textrm{ a.e. in }\{\xi \ge X(\eta),\tau'(x_{\xi,Y(\xi)})< 1\},
\end{array}
\]
and
\[
\begin{array}{ll}
\tilde u_\eta(\xi, \eta)  = \tilde w_\eta(X(\eta),\eta)&\quad\textrm{ a.e. in }\{\eta < Y(\xi)\}\cup\{\eta \ge Y(\xi),\tau'(x_{X(\eta),\eta}) = -1\}\\
\tilde u_\eta(\xi, \eta)  = -\tilde w_\eta(X(\eta),\eta)&\quad\textrm{ a.e. in }\{\eta \ge Y(\xi),\tau'(x_{X(\eta),\eta})> -1\}.
\end{array}
\]

This implies that
\[
|\tilde u_\xi(\xi, \eta)|  = |\tilde w_\xi(\xi, Y(\xi))|\quad\textrm{and}\quad|\tilde u_\eta(\xi, \eta)|  = |\tilde w_\eta(X(\eta),\eta)|
\]
for a.e. $(\xi, \eta)$ with $\xi+\eta > 0$, and the result follows. 
\end{proof}

The previous proposition establishes a very clear intuition of what the solution to the obstacle problem looks like (or, more precisely, how the derivatives in the directions of the characteristics look like). In particular, it establishes a partition of the region $I$ into three different parts, depending on which directional derivative in $\xi$ and $\eta$ flips sign with respect to the free wave equation, see Figure~\ref{fig.2}. 

\begin{figure}
\centering
\includegraphics{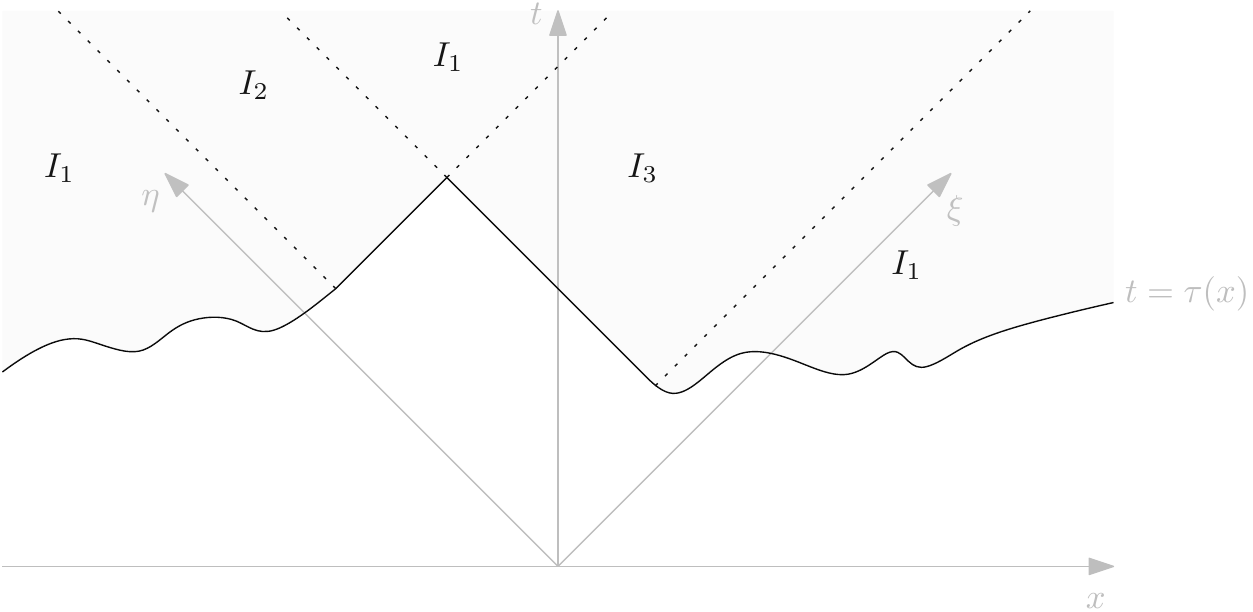}
\caption{$I = I_1\cup I_2\cup I_3$, where in $I_1$ both derivatives in $\xi$ and $\eta$ are flipped along characteristics, in $I_2$ only the derivative in $\eta$ is flipped, and in $I_3$ only the derivative in $\xi$.}
\label{fig.2}
\end{figure}

As a consequence of the previous proposition we immediately obtain that if we start from Lipschitz initial data (namely, $u_x$ and $u_t$ bounded), the solution remains Lipschitz at all times.

\begin{cor}
\label{cor.Lip}
Let $u_0\in W^{1, \infty}_{\rm loc}(\R)$ and $u_1\in L^\infty_{\rm loc}(\R)$. Assume that $u_0 \ge 0$, and that $u_1 \ge 0$ a.e. inside $\{u_0 = 0\}$. Then, there exists a unique solution $u\in L^\infty_{{\rm loc}, t} ((0, \infty); W^{1, \infty}_{{\rm loc}, x} (\R))\cap W^{1,\infty}_{{\rm loc}, t} ((0, \infty); L^{ \infty}_{{\rm loc}, x} (\R))$. Moreover, if $u_0\in W^{1, \infty}(\R)$ and $u_1\in L^\infty(\R)$ then 
\[
\begin{split}
\|u_x(\cdot, t) - u_t(\cdot, t)\|_{L^\infty(\R)} & = \|u_x(\cdot,0) - u_t(\cdot,0)\|_{L^\infty(\R)},\\
\|u_x(\cdot, t) + u_t(\cdot, t)\|_{L^\infty(\R)} & = \|u_x(\cdot,0) + u_t(\cdot,0)\|_{L^\infty(\R)}\qquad \text{for a.e. $t \geq 0$}.
\end{split}
\]
In particular, 
\[
\frac{1}{\sqrt{2}} \|\nabla_{x,t} u(\cdot, 0)\|_{L^\infty(\R)}\le \|\nabla_{x,t} u(\cdot, t)\|_{L^\infty(\R)} \le \sqrt{2} \|\nabla_{x,t} u(\cdot, 0)\|_{L^\infty(\R)}
\]
for a.e. $t\ge 0$. 
\end{cor}
\begin{proof}
The first part is a restatement of Proposition~\ref{prop.Lip}.

For the second part, from Proposition~\ref{prop.Lip} and by changing of variables,  
\begin{align*}
|u_x(x, t)|^2 + |u_t(x, t)|^2 & = |\tilde u_\xi(x, t)|^2 + |\tilde u_\eta(x, t)|^2  \\
& = |\tilde u_\xi(x+t, 0)|^2 + |\tilde u_\eta(x-t, 0)|^2\\
& \le |u_x(x+t, 0)|^2+|u_t(x+t, 0)|^2 +|u_x(x-t, 0)|^2+|u_t(x-t, 0)|^2.
\end{align*}
In particular, for every compact set $(-L,L)\subset\subset \R$, we are showing that 
$$
\|\nabla_{x,t} u(\cdot, t)\|^2_{L^\infty((-L,L))} \le 2\|\nabla_{x,t} u(\cdot, 0)\|^2_{L^\infty((-L-t,L+t))}.
$$
Letting $L\to +\infty$, this yields
\begin{equation}
\label{eq.onedir}
\|\nabla_{x,t} u(\cdot, t)\|^2_{L^\infty(\R)} \le 2\|\nabla_{x,t} u(\cdot, 0)\|^2_{L^\infty(\R)}.
\end{equation}
On the other hand, notice that solutions to the obstacle problem are time-reversible: if $u(x, t)$ is a solution to an obstacle problem for the wave equation ($\min\{\square u, u\} = 0$ and \eqref{eq.consene} holds), then $v(x, t) = u(x, T-t)$ is also a solution to the obstacle problem for the wave equation. Thus, applying \eqref{eq.onedir} to $u(\cdot, t-\cdot)$ we obtain the desired result. 
\end{proof}

\section{The double obstacle problem}

The previous section not only establishes that the Lipschitz constants of solutions are preserved at all times (in the characteristic variables), but also shows that, when starting from Lipschitz data, the whole problem can be treated at a local level. In particular, this allows us to treat the double obstacle case, and the same reasoning as before yields that solutions to the double obstacle problem preserve the Lipschitz constant (since Proposition~\ref{prop.Lip} still holds). 

That is, consider now that the solution not only is forced to remain above an obstacle $\varphi \equiv 0$, but also is enclosed to be below $\bar\varphi\equiv 1$. Locally, when hitting $\bar\varphi$ the solution is behaving like an obstacle problem for the wave equation (with reverse displacement from the previous configuration). Thus, we obtain the validity of the following: 

\begin{prop}
\label{prop.dlip}
Let $u_0\in W^{1, \infty}_{\rm loc}(\R)$ and $u_1\in L^\infty_{\rm loc}(\R)$. Assume that $0\le u_0 \le  1$, and that $u_1 \ge 0$ a.e. inside $\{u_0 = 0\}$ and $u_1 \le 0$ a.e. inside $\{u_0 = 1\}$. Then, there exists a unique solution $u\in L^\infty_{{\rm loc}, t} ((0, \infty); W^{1, \infty}_{{\rm loc}, x} (\R))\cap W^{1,\infty}_{{\rm loc}, t} ((0, \infty); L^{\infty}_{{\rm loc}, x} (\R))$ to
\begin{equation}
\label{eq.pb2}
\left\{
\begin{array}{rcll}
\min\{\square u, u\} & = & 0& \textrm{ in } (\R\times [0, \infty))\cap\{u < 1\}\\
\min\{-\square u, 1-u\} & = & 0& \textrm{ in } (\R\times [0, \infty))\cap\{u > 0\}\\
u(\cdot, 0) & = &u_0& \textrm{ in } \R\\
u_t(\cdot, 0) & = &u_1& \textrm{ a.e. in } \R.
\end{array}
\right.
\end{equation}
such that \eqref{eq.consene} holds in the sense of distributions. Moreover, if $u_0\in W^{1, \infty}(\R)$ and $u_1\in L^\infty(\R)$ then 
\[
\begin{split}
\|u_x(\cdot, t) - u_t(\cdot, t)\|_{L^\infty(\R)} & = \|u_x(\cdot, 0) - u_t(\cdot, 0)\|_{L^\infty(\R)},\\
\|u_x(\cdot, t) + u_t(\cdot, t)\|_{L^\infty(\R)} & = \|u_x(\cdot, 0) + u_t(\cdot, 0)\|_{L^\infty(\R)}.
\end{split}
\]
In particular, 
\[
\frac{1}{\sqrt{2}} \|\nabla_{x,t} u(\cdot, 0)\|_{L^\infty(\R)}\le \|\nabla_{x,t} u(\cdot, t)\|_{L^\infty(\R)} \le \sqrt{2} \|\nabla_{x,t} u(\cdot, 0)\|_{L^\infty(\R)}
\]
for a.e. $t\ge 0$. 
\end{prop}
\begin{proof}
Since the initial condition and the solution is locally Lipschitz, the construction of the solution by Schatzman as explained above and its properties can also be performed in this case, locally. In particular, Proposition~\ref{prop.Lip} also holds and the desired result follows as in the proof of Corollary~\ref{cor.Lip}.
\end{proof}

\section{An explicit solution by Bamberger and Schatzmam}
\label{sect:counterex}

In \cite{BS83}, Bamberger and Schatzman established an explicit formula for the solution $u$ to the obstacle problem (with zero obstacle) in terms of the free wave solution $w$. Such explicit formula is given by 
\begin{equation}
\label{eq.wrong}
u(x, t) = w(x, t) +2\sup_{(x',t')\in T_{x,t}^-} (w(x',t'))^-,
\end{equation}
where $r^-$ denotes the negative part, that is $r^- := \sup\{-r,0\}$. Unfortunately, as we show here, such formula cannot hold true. 

To see that, consider the problem with initial conditions given by $u_0(x) = \frac12$ and $u_1(x) =\sin(x)$. The free wave solution is explicit, and is given by 
\[
w(x, t) = \frac12 +\sin(x)\sin(t).
\]
In particular, $w$ is bounded and therefore, if the formula above was correct, we would deduce that 
\begin{equation}
\label{eq.contr}
u\le \frac92. 
\end{equation}
Nonetheless, the solution to the obstacle problem in this case goes to infinity as $t\to \infty$, a contradiction. Indeed, recalling \eqref{eq.sol}, the solution is given by 
\[
u(x, t) = w(x, t) + (\mathcal{E}\ast \mu(w))(x, t)
\]
with 
\[
\langle \mu(w), \psi\rangle = -2\int (1-\tau'(x)^2) \,w_t(x, \tau(x) ) \,\psi(x, \tau(x))\, dx \qquad \forall\,\psi \in C_c(\R\times [0,+\infty)).
\]
That is, 
\[
u(x, t) = w(x, t) -\int_{\{t-\tau(z)\ge |x-z|\}}  (1-\tau'(z)^2) \,w_t(z, \tau(z) ) \, dz.
\]
Notice that, whenever $|\tau'(z)|< 1$, it holds $w_t(z, \tau(z))\le 0$
(since the free wave is touching the obstacle $\varphi=0$ coming from above). In addition, at $x_\circ = -\frac{\pi}{2}$ we have $\tau(x_\circ) = \frac{\pi}{6}$, $\tau'(x_\circ) = 0$, and $w_t(x_\circ, \tau(x_\circ)) = \frac{1-\sqrt{3}}{2} < 0$. In particular, by continuity, there exists $c_\circ>0$ such that
\[
\int_{-\pi}^\pi  (1-\tau'(z)^2) w_t(z, \tau(z) ) \, dz = -c_\circ < 0. 
\]
On the other hand, since the solution is $2\pi$-periodic and $\tau$ is 1-Lipschitz, $\tau\le \frac{7\pi}{6}$. Thus,
\[
u(x, t) \ge w(x, t) -\int_{\{|z-x|\le t-7\pi/6\}}  (1-\tau'(z)^2) w_t(z, \tau(z) ) \, dz.
\]
Hence, by $2\pi$-periodicity of the integrand, if $t \ge 2k\pi+\frac{7\pi}{6}$ for some $k\in \N$ then 
\[
u(x, t) \ge w(x, t) +kc_\circ,
\]
and letting $t\to \infty $ (and, therefore, $k\to \infty$) we deduce that $u(x, t) \to \infty$ as $t\to\infty$. This is in contradiction with \eqref{eq.contr}, and thus \eqref{eq.wrong} cannot hold.

\begin{rem}
\label{rem.inf}
In fact, the previous argument shows that solutions with periodic initial data and with an ``active'' obstacle (namely, the solutions hit at some moment the obstacle with positive velocity)
will grow to infinity as time goes by, regardless of the form of the solution. 
\end{rem}

\end{document}